\documentclass[11pt, reqno]{amsart}
\usepackage{graphicx, amssymb, amsmath, amsthm}
\usepackage{epsfig}
\usepackage{hyperref}
\numberwithin{equation}{section}

\usepackage{tikz}
\usepackage{here}

\usepackage{pifont}

\usetikzlibrary{matrix,arrows}
\usetikzlibrary{shapes}
\usetikzlibrary{calc}
\usetikzlibrary{arrows}
\usetikzlibrary{decorations.pathreplacing,decorations.markings}

\newtheorem{theorem}{Theorem}[section]

\newtheorem{lemma}[theorem]{Lemma}

\newtheorem{corollary}[theorem]{Corollary}

\newtheorem{proposition}[theorem]{Proposition}

\theoremstyle{definition}
\newtheorem{definition}[theorem]{Definition}
\newtheorem{remark}[theorem]{Remark}

\makeatletter
\newcommand{\Extend}[5]{\ext@arrow0099{\arrowfill@#1#2#3}{#4}{#5}}
\makeatother

\begin{document}
\title[Slow decay of Scalar curvature]{simply-connected open $3$-manifolds with slow decay of positive scalar curvature}

\author[J. Wang]{Jian Wang}
\address{Universit\'e Grenoble Alpes, Institut Fourier, 100 rue des maths, 38610 Gi\`eres, France}
\email{jian.wang1@univ-grenoble-alpes.fr}
 \maketitle
 \begin{abstract} The goal of this paper is to investigate the topological structure of open simply-connected 3-manifolds whose scalar curvature has a slow decay at infinity. In particular, we show that the Whitehead manifold does not admit a complete metric, whose scalar curvature decays slowly, and in fact that any contractible complete 3-manifolds with such a metric is  diffeomorphic to $\mathbb{R}^{3}$. Furthermore, using this result, we prove that any open simply-connected 3-manifold $M$ with $\pi_{2}(M)=\mathbb{Z}$ and a complete metric as above, is diffeomorphic to $\mathbb{S}^{2}\times \mathbb{R}$.  
\end{abstract} 
\section{introduction}
Thanks to  G.Perelman's proof of W.Thurston's Geometrisation conjecture in \cite{P1,P2,P3}, the topological structure of compact 3-manifolds is now well understood. However, it is known from the early work \cite{Wh} of J.H.C Whitehead that the topological structure of non-compact 3-manifolds is much more complicated. For example, there exists a contractible open 3-manifold, called the Whitehead manifold, which is not homeomorphic to $\mathbb{R}^{3}$. An  interesting question in differential geometry is whether the Whitehead manifold admits a complete metric with positive scalar curvature. 

The study of manifolds of positive scalar curvature has a long history. Among many results, we mention the topological classification of compact manifolds of positive  scalar curvature and the Positive Mass Theorem.  There are two methods which have achieved many breakthroughs: minimal hypersurfaces  and K-theory.

The K-theory method is pioneered by A.Lichnerowicz and is systemically developed by M.Gromov and H.Lawson in \cite{GL}, based on the Atiyah-Singer Index theorem in \cite{AS}. Furthermore,  combined with some results about the Novikov conjecture, S.Chang, S.Weinberger and G.L.Yu \cite{CWY} investigated the topological structure of open 3-manifolds with uniformly positive scalar curvatures and finitely generated fundamental groups. Precisely, they proved that any contractible 3-manifold whose scalar curvature is bounded away from zero is $\mathbb{R}^{3}$, which  implies that the Whitehead manifold does not admit a metric with uniformly positive scalar curvature.\par
The origin of the minimal hypersurfaces method is the article of R.Schoen and S.T.Yau \cite{SY}. For open 3-manifolds, there are many applications,  such as Positive Mass Theorem\cite{SY1,SY2} and 3-dimensional Milnor Conjecture\cite{L}. In \cite{GL}, Gromov and Lawson applied this method to open 3-manifolds. \par In this paper, we  extend Gromov-Lawson's Theorem\cite{GL} to open 3-manifolds whose scalar curvature has a decay at infinity:

\begin{theorem}\label{A}
Assume that $(M^{3}, g)$ is a contractible complete 3-manifold. Let $0\in M$ and $r(x)$ be the distance function from $x$ to $0$. If there exists a number $\alpha\in [0,2)$ such that 
$$\liminf\limits_{r(x) \rightarrow \infty }r^{\alpha}(x)\mathrm{Scal}(x)>0,$$
then $M^{3}$ is diffeomorphic to  $\mathbb{R}^{3}$.
\end{theorem}

Our main tool  comes from the solution to the Plateau Problem. 
Let us give a brief review of the existence of the solution to the Plateau problem and their regularity. In the case of $\mathbb{R}^{3}$, the existence is due to Douglas \cite{D} and Rad\'o \cite{R} : For any smooth curve $\gamma$ in $\mathbb{R}^{3}$, there exists a surface with minimal area, spanning $\gamma$, which is parametrized by a disc $D^{2}$. Furthermore, Osserman \cite{O} and Gulliver \cite{G} proved that this solution has no interior branch point. In 1948, Morrey \cite{M1,M2} devised a new method to solve the Plateau problem for a map from a disc to a ``homogeneously regular" Riemannian manifold. In addition, Osserman and Gulliver's arguments show that Morrey's solution also has no interior branch point. \par

To sum up, for any null-homotopic smooth curve $\gamma$ in a complete Riemannian manifold $(M^{3}, g)$, there exists a continuous map $f: D^{2}\rightarrow M$ such that  $f$ induces a homomorphism between $\partial D$ and $\gamma$ and the interior of $f$ is a minimal immersion.\par

Let us explain the scheme of Gromov-Lawson's proof in \cite{GL}. From J.Stalling's results in \cite{S},  it is sufficient to prove that $M$ is simply-connected at infinity. We argue by contradiction and suppose that $M$ is not simply-connected at infinity. Namely, there exists a compact set $K$ in $M$, satisfying that for any geodesic ball $B_{R}(0)$ containing $K$, there is a smooth closed curve $\gamma$ in $M\setminus B_{R}(0)$, which is not null-homotopic in $M\setminus K$. Choosing $R$ large enough and  considering the solution $S$ to the Plateau problem for $\gamma$, $S$ can not be contained in a``small" neighborhood of some circle(s), which contradicts to a result of Rosenberg in \cite{RO}. Before stating  Rosenberg's result, we firstly define the notion of stable $H$-surface.

\begin{definition} Let $(M, g)$ be a 3-dimensional Riemannian manifold and $\Sigma$ a surface immersed in $(M, g)$. $\Sigma$ is said to be a \emph{$H$-surface} if the mean curvature of $\Sigma$ is constant and equals to $H$, where $H\in \mathbb{R}$. Furthermore, we consider the operator $L$ of $\Sigma$  defined as 
$$L=\Delta_{\Sigma}+|A|^{2}+Ric_{M}(\bf{n})$$
where $|A|^{2}$ is the square length of the second fundamental form of $\Sigma$ in $M$, $\bf{n}$ is a unit normal vector field along $\Sigma$. We say that $\Sigma$ is a \emph{stable $H$-surface}  if 
$$-\int_{\Sigma}uL(u)\geq 0$$ for any smooth function $u$ with compact support over $\Sigma$.
\end{definition}
In \cite{RO}, applying the stable condition, Rosenberg obtained the following:

\begin{theorem}\cite{RO}\label{B} 
Let $(N, g)$ be a complete Riemannian 3-manifold satisfying 
$$3H^{2}+S(x)\geq c,~~~\text{for any}~x\in N,$$ where $S(x)$ is the scalar curvature, $H\in \mathbb{R}$ and $c>0$. Suppose that $\Sigma$ is a stable $H$-surface immersed in $N$. Then one has for any $x\in \Sigma$
$$d_{\Sigma}(x, \partial \Sigma)\leq 2\pi/(3c)^{1/2}, $$
where $d_{\Sigma}$ is the intrinsic distance in $\Sigma$.
\end{theorem}
\begin{remark}\label{R} From the proof of Theorem \ref{B} in \cite{RO}, we do not need the condition that the scalar curvature $S(x)$ is uniformly bounded. 
\end{remark}
When trying to generalize Gromov-Lawson's arguments to contractible 3-manifolds whose scalar curvature decay at infinity, one encounters  an obstacle--the lack of uniform lower-bound of the scalar curvature on the surface $S$ spanning the curve $\gamma$. However, this can be overcome by choosing a new curve $\sigma$ near the boundary of $B_{R/2}(0)$, which is homotopic to $\gamma$ in $M\setminus K$. More precisely,  using the regularity property of $S$, after a small deformation,  the intersection of $S$ and the boundary of $B_{R/2}(0)$ is some circle(s). Therefore, the intersection can be chosen as a curve $\sigma$, which satisfies the required property. This leads to Theorem \ref{A}, by applying Gromov-Lawson's argument\cite{GL} to this curve $\sigma$.

\section{The proof of theorem \ref{A}}
Before the proof of Theorem \ref{A}, we introduce an important notion and recall a classical result about open 3-manifolds. 

\begin{definition} A topological space $X$ is said to be \emph{simply-connected at infinity}, if for any compact set $C$, there exists a compact set $V$ containing $C$, such that the induced map $i_{\ast}:\pi_{1}(X\setminus V)\rightarrow\pi_{1}(X\setminus C)$ is trivial, where $i: X\setminus V \rightarrow X\setminus C$ is an  inclusion map. 
\end{definition}

It is well-known from \cite{Wh} that  the Whitehead manifold is not simply-connected at infinity.  In fact, there is a unique non-compact  $3$-manifold which is  both simply-connected at infinity and contractible.

\begin{theorem}\label{S}[Stallings] \cite{S} Let $X$ be a contractible 3-manifold, then $X$ is homeomorphic to $\mathbb{R}^{3}$ if and only if $X$ is simply-connected at infinity.
\end{theorem}

Here is the idea of proof of Theorem \ref{A}. In the following, we assume that $(M, g)$ is a contractible 3-manifold as in Theorem \ref{A}. According to Theorem \ref{S}, it is sufficient to show that $M$ is simply-connected at infinity. From the assumption about the scalar curvature in Theorem \ref{A}, there are two positive constants $C$ and $R_{0}$, such that for $r(x)\geq R_{0}$, one has, $$\mathrm{Scal}(x)\geq \frac{C}{{r(x)}^{\alpha}}.$$ Compared with the argument of Gromov-Lawson in  \cite{GL}, our main difficulty is the lack of the uniform lower-bound of the scalar curvature. However, using Theorem \ref{B}, we obtain the following:

\begin{proposition}\label{D} 
If $R>2 \max\left\{R_{0}, (\frac{4^{1+\alpha/2}\pi}{(3C)^{1/2}})^{\frac{2}{2-\alpha}}\right\}$,  then the induced map $\pi_{1}(M\setminus B_{4R}(0))\rightarrow \pi_{1}(M\setminus B_{R}(0))$ is trivial.
\end{proposition}
Then, due to Proposition \ref{D}, $M$ is simply-connected at infinity, which implies that $M$ is homeomorphic to $\mathbb{R}^{3}$. It implies Theorem \ref{A}. \par
We now prove Proposition \ref{D}. Let us consider a smooth closed curve $\gamma$ in $M\setminus B_{4R}(0)$. Because $\gamma$ may be far away from the compact set $B_{R}(0)$, there does not exist a ``good" estimate of the lower-bound of the scalar curvature along $\gamma$. In order to overcome it, we establish the following lemma:

\begin{lemma}\label{C}
For any smooth circle $\gamma$ in $M\setminus B_{4R}(0)$, exactly one of the following holds:  
\begin{itemize}
\item $\gamma$ is  contractible in $M\setminus B_{R}(0)$, 
\item for any $\epsilon>0$ and any $R<Q<4R$, there exists a curve $\hat{\sigma}$ in $ B_{Q+\epsilon}(0)\setminus B_{Q}(0)$, which is not contractible in $M\setminus B_{R}(0)$.
\end{itemize}
\end{lemma}

\begin{proof} Let $\hat{f}: D^{2}(1)\rightarrow M$ be the solution to the Plateau problem  for $\gamma$, where $D^{2}(1)$ is the unit disc. Then $\hat{f}|_{\partial{D^{2}(1)}}$ is a homeomorphism between $\partial {D^{2}(1)}$ and $\gamma$. Furthermore, the interior of $\hat{f}$ is an immersion.  \par
 Suppose that $\gamma$ is non-contractible in $M\setminus B_{R}(0)$. Therefore, $\hat{f}(D^{2}(1))\cap \partial B_{Q}(0)$ is nonempty. Furthermore, $\hat{f}^{-1}(\hat{f}(D^{2}(1))\cap B_{Q}(0))$ belongs to the interior of $D^{2}(1)$. Then, after a small variation of $\hat{f}$ in the interior of $D^{2}(1)$, there exists a map $f:D^{2}(1)\rightarrow M$, such that 
 \begin{itemize}
 \item the interior of  $f$ is still an immersion, 
 \item  $f|_{\partial D^{2}(1)}$ is a homoemorphic map from $\partial D^{2}(1)$ to $\gamma$, 
 \item $f$ intersects transversally with $\partial B_{Q}(0)$.
 \end{itemize}
 Therefore, the pre-image of $f(D^{2}(1))\cap \partial B_{Q}(0)$ is a 1-dimensional  compact submanifold in $D^{2}(1)$. That is to say: $f^{-1}(f(D^{2}(1))\cap \partial B_{Q}(0))$ is a disjoint union of some circle(s) in $D^{2}(1)$.\par
  
 Let $f^{-1}(B_{Q}(0))$ be the disjoint union of $\{\gamma_{i}\}_{i\in I}$, where each $\gamma_{i}$ is diffeomorphic to a circle in $D^{2}(1)$ and $I$ is a finite set. Let $D_{i}$ be the unique disc bounded by $\gamma_{i}$ in $D^{2}(1)$. Let us consider the set $\{D_{i}\}_{i\in I}$ and define the partially ordered relationship on it, induced by inclusion. 
 For any maximal element $D_{j}$ in $(\{D_{i}\}_{i\in I}, \subset )$,  $\gamma_{j}=\partial D_{j}$ is defined as an outmost circle.

Let $\{\gamma_{j}\}_{j\in I_{0}}$ be outermost circles in $\{\gamma_{i}\}_{i\in I}$, where $I_{0}\subset I$. For each outmost circle $\gamma_{j}$, we assume that ${\sigma_{j}}$ is the boundary of a tubular neighborhood of $\gamma_{i}$ in $D^{2}(1)\setminus \cup_{i\in I} D_{i}$, contained in $f^{-1}(B_{Q+\epsilon}(0)\setminus B_{Q}(0))$ and $\hat{D}_{j}$ is the unique disc bounded by $\sigma_{j}$ in $D^{2}(1)$.\par

$\bold{Claim}$: some element in $\{f(\sigma_{j})\}_{j\in I_{0}}$ is not contractible  in $M\setminus B_{R}(0)$.\par
We argue it by contradiction. Suppose that each $f(\sigma_{j})$ is contractible in $M\setminus B_{R}(0)$. In other words, for each $j\in I_{0}$ there exists a continuous map $g_{j}:D^{2}(1)\rightarrow M$, such that  (1) $g_{j}(D^{2}(1))\cap B_{R}(0)=\emptyset$; (2) $g_{j}|_{\partial D^{2}(1)}$ is a homomorphism from $\partial D^{2}(1)$ to $f(\sigma_{j})$.\par

After changing the coordinate over $D^{2}(1)$, we may suppose that for each $j\in I_{0}$, $\hat{D}_{j}$ is a disc with center at $x_{j}$ and radius $r_{j}$. We will construct a new map $g:D^{2}(1)\rightarrow M$, such that (1) $g|_{\partial D^{2}(1)}$ is a homeomorphic map from $\partial D^{2}(1)$ to $\gamma$; $g(D^{2}(1))\cap B_{R}(0)=\emptyset$. It implies that $\gamma$ is contractible in $M\setminus B_{R}(0)$, which contradicts our assumption that $\gamma$ is not contractible in $M\setminus B_{R}(0)$.\par
 Let us describe the construction of $g$ as follows:

 \begin{equation}
 g(x)=\left\{  
\begin{array}{cc}
f(x),&\quad x\in D^{2}(1)\setminus \cup_{j\in I_{0}}\hat{D}_{j}\\
g_{j}(\frac{x-x_{j}}{r_{j}}),&\quad x\in \hat{D}_{j}.\\
\end{array}
\right.
\end{equation}
 $g$ is a required map described as above. We finish the proof of the claim. \par
 We may suppose that $f(\sigma_{j_{0}})$ is non-contractible in $M\setminus B_{R}(0)$ and choose $\hat{\sigma}=f(\sigma_{j_{0}})$. It is the required candidate in the assertion. This completes the proof.  
\end{proof}

\begin{remark}\label{L}

The proof of Lemma \ref{C}, $\hat{f}$ just requires that $\hat{f}$ is an immersion. The reason is described below:

If $\hat{f}$ is an immersion, we can deform it to obtain  a new immersion $f$ satisfying that $f$ intersects $B_{Q}())$ transversally. It is sufficient for our proof
\end{remark}

 We now give the proof of Proposition \ref{D}.
\begin{proof}[Proof of Proposition \ref{D}] We  prove it by contradiction. First, we assume that for some $R>2 \max\{R_{0}, (\frac{4^{1+\alpha/2}\pi}{(3C)^{1/2}})^{\frac{2}{2-\alpha}}\}$, there is a curve $\gamma\subset M\setminus B_{4R}(0)$ such that $\gamma$ is nontrivial in $\pi_{1}(M\setminus B_{R}(0))$.\par
We take $Q=2R$ and $\epsilon=1$ and apply Lemma \ref{C} to $\gamma$. There is a non-contractible curve $\hat{\sigma} \subset B_{2R+1}(0)\setminus B_{2R}(0)$ in $M\setminus B_{R}(0)$.\par 
Let  $f:D^{2}(1)\rightarrow M$ be the solution to the Plateau problem for the circle $\hat{\sigma}$ in M. Then $f(D^{2}(1))\cap \partial B_{R}(0)$ is non-emtpy, since $\hat{\sigma}$ is non-contractible  in $M\setminus B_{R}(0)$.  Let us consider the set $\Sigma:=f(D^{2}(1))\cap (B_{3R}(0)\setminus B_{R}(0))$. By the assumption of the scalar curvature, i.e. $\liminf\limits_{r(x) \rightarrow \infty }r^{\alpha}(x)\mathrm{Scal}(x)>0$, one has,$$\mathrm{Scal}(x)\geq \frac{C}{(4R)^{\alpha}},~~\text{over} ~~~~~\Sigma.$$
By Theorem \ref{B} and Remark \ref{R}, we deduce that

 $$\Sigma ~\text{is contained in the}~\frac{2(4R)^{\alpha/2}\pi}{(3C)^{1/2}}\text{-neighborhood of } \partial \Sigma.$$
However, $\partial \Sigma=\hat{\sigma}\amalg (\partial B_{R}(0)\cap \Sigma)$ is contained in a union of $M\setminus B_{2R}(0)$ and the closure of $B_{R}(0)$. The above fact gives 
\begin{equation*} R\leq 4\frac{(4R)^{\alpha/2}\pi}{(3C)^{1/2}}\end{equation*}That is to say, $R\leq  \left(\frac{4^{1+\alpha/2}\pi}{(3C)^{1/2}}\right)^{\frac{2}{2-\alpha}}< 2(\frac{4^{1+\alpha/2}\pi}{(3C)^{1/2}})^{\frac{2}{2-\alpha}}$. This contradicts the choice of $R$. 
\end{proof}

As a corollary, we have the following
\begin{corollary}The Whitehead manifold does not admit a complete metric with $\liminf\limits_{r(x) \rightarrow \infty }r^{\alpha}(x)\mathrm{Scal}(x)>0$, where $r$ is a distance function from a given point and $\alpha\in [0,2)$.\end{corollary}

\section{Application}
In this section, we use Theorem \ref{A} to prove that any complete open 3-manifold $M$ such that (1) $\pi_{1}(M)=\{0\}$ and $\pi_{2}(M)=\mathbb{Z}$; (2) its scalar curvature decays slowly at infinity, is homeomorphic to $\mathbb{R}\times \mathbb{S}^{2}$. Let us review the sphere theorem and some classical applications of the sphere theorem. 

\begin{theorem}[Sphere Theorem \cite{PA, BA}]\label{E}  Any orientable 3-manifold $M$ with non-trivial $\pi_{2}(M)$ has an embedded sphere. \end{theorem}
The non-trivial embedded sphere plays a crucial role in the prime decomposition of 3-manifolds. In a simply-connected 3-manifold, the non-trivial embedded sphere always separates the $3$-manifold into two connected components.

\begin{lemma}\label{F} Let $M^{3}$ be a simply-connected open manifold and $S\subset M$  a non-trivial sphere, then $M\setminus S$ is a disjoint union of two non relatively compact components.  
 \end{lemma}
 
 \begin{proof} First, we claim that $S$ is a separating sphere in $M$. If $M\setminus S$ is connected, then there exists a circle $\alpha \subset M$, such that $\alpha$ intersects $S$ transversally at one point. Hence, the intersection number $(S, \alpha)$ is $1$ or $-1$.  However, the intersection number between $S$ and $\alpha$ is zero, since $\alpha$ is contractible in $M$ and the intersection number is a homotopic invariant. It leads to a contradiction. Therefore, $S$ separates $M$ into two components.\par
 
 Second, we will show that each component of $M\setminus S$ is non relatively compact. If one component, denoted by $M'$,  is relatively compact, we use the Van Kampen Theorem and $\pi_{1}(M)=\{1\}$  to obtain that $M'$ is also simply connected. We define a 3-manifold $M''$ by gluing a 3-ball along the boundary sphere $\partial M'$. Hence, by the Van Kampen theorem, the compact 3-manifold $M''$ is also simply connected. By the Poincar\'e conjecture \cite{P1, P2, P3}, $M''$ is a 3-sphere. Therefore, $M'$ is homeomorphic to a 3-ball, which contradicts  our assumption that $S$ is non-trivial in $\pi_{2}(M)$.\end{proof}
 
The second homotopy group is an important homotopic invariant for the classification of CW complexes. We will use the following lemma frequently.

\begin{lemma}{\label{H}} An open 3-manifold $M$ with $\pi_{1}(M)=\pi_{2}(M)=0$ is contractible. 
\end{lemma}
\begin{proof}Because $M$ is an open manifold, the 0-dimensional cohomology group of $M$ with compact support, denoted by $H^{0}_{c}(M,\mathbb{Z})$, is trivial.  The Poincar\'e duality gives $H^{0}_{c}(M,\mathbb{Z})\cong H_{3}(M, \mathbb{Z})=\{0\}$. Since $\pi_{1}(M)=\pi_{2}(M)=\{0\}$, we use the Hurewicz Theorem to see that  $\pi_{3}(M)\cong H_{3}(M)=\{0\}$. Then, using the Hurewicz Theorem inductively, we have that $\pi_{n}(M)\cong H_{n}(M)=0$, for any $n\geq 3$. By the Whitehead Theorem, $M$ is contractible.
\end{proof}

We now consider a simply-connected open $3$-manifold $M$ with nontrivial second homotopy group. From Theorem \ref{E} and Lemma \ref{F}, there exists a non-trivial sphere $S$ seperating $M$ into two non relatively compact parts: $M_{1}$ and $M_{2}$. We define $M^{0}_{i}$ as a union of $B^{3}$ and $M_{i}$ along $S$, where $B^{3}$ is a 3-ball. Then $M$ can be viewed as the connected sum of $M^{0}_{0}$ and $M^{0}_{1}$. \par
The pair $(M, \bar{M_{i}})$ is a CW pair, where $\bar{M_{i}}$ is the closure of $M_{i}$ in $M$, for $i\in I$, where $I=(\mathbb{Z}/2\mathbb{Z}, +)$. For each CW pair $(M, \bar{M_{i}})$, one has a continuous map $f_{i}: M\rightarrow M/\bar{M_{i}}$. Because $\partial M_{i}$ is a sphere $S$, $M/\bar{M_{i}}$ is homeomorphic to $M^{0}_{i+1}$, for $i\in I=(\mathbb{Z}/2\mathbb{Z},+)$. Therefore, $f_{i}$ can be viewed as  a continuous map from $M$ to $M^{0}_{i+1}$. Furthermore, the induced map $(f_{i})_{\ast}: \pi_{2}(M)\rightarrow\pi_{2}(M^{0}_{i+1})$ verifies the following:

\begin{proposition}{\label{J}}
For each $i\in I$, the induced map $(f_{i})_{\ast}:\pi_{2}(M)\rightarrow \pi_{2}(M^{0}_{i+1})$ is  a surjective map with nontrivial kernel. 
\end{proposition} 
\begin{proof}In the above statement, $S$ is a non-trivial sphere in $M$. $M^{0}_{i+1}$ is obtained by gluing a $3$-ball $B^{3}$ along $S$ to $M_{i+1}$, i.e. $M_{i+1}^{0}=M_{i+1}\bigcup_{S}B^{3}$, for $i\in I$. \par
The image of $S$ is bounded by a $3$-ball, which implies that $[f_{i}(S)]$ is trivial in $\pi_{2}(M^{0}_{i+1})$. Therefore, the kernel of $(f_{i})_{\ast}$ is non-trivial.\par
  For each $i\in I$, any continuous map $g:\mathbb{S}^{2}\rightarrow M^{0}_{i+1}$ is homotopic to another continuous map $g':\mathbb{S}^{2}\rightarrow M_{i+1}$ in $M^{0}_{i+1}=M_{i+1}\bigcup_{S}B^{3}$. Let $i:M_{i+1}\rightarrow M$ be the inclusion map, $i\circ g'$  is a map from $\mathbb{S}^{2}$ to $M$ and $(f_{i})_{\ast}([g'])=[g]$ in $\pi_{2}(M^{0}_{i+1})$. Therefore, $(f_{i})_{\ast}$ is surjective. 
\end{proof}

In the following, we will use the sphere theorem to analyze the topological structure of a simply-connected open 3-manifold $M$ with $\pi_{2}(M)=\mathbb{Z}$. Together with the surgery as described above, we  prove :
\begin{theorem} \label{G}
Assume that $(M^{3}, g)$ is a simply-connected open 3-manifold with $\pi_{2}(M)=\mathbb{Z}$. Let $0\in M$ be a  point and $r(x)$ a distance function from $x$ to $0$. If there exists a real number $\alpha\in [0,2)$, such that, 
$$\liminf\limits_{r(x) \rightarrow \infty }r^{\alpha}(x)\mathrm{Scal}(x)>0,$$then $M^{3}$ is diffeomorphic to $\mathbb{R}\times \mathbb{S}^{2} $.
\end{theorem}

\begin{proof}
By the Sphere theorem , there exists a non-trivial embedded sphere $S$ cutting $M$ into two non relatively compact components $M_{0}$ and $M_{1}$. $M$ can be viewed as the connected sum of two simply-connected 3-manifolds $M^{0}_{0}$ and $M^{0}_{1}$, where $M^{0}_{0}=M_{0}\bigcup_{S}B^{3}$ and $M^{0}_{1}=M_{1}\bigcup_{S} B^{3}$.\par 
Furthermore, we will construct a metric over $M^{0}_{i}$ satisfying the curvature condition as assumed in Theorem \ref{A}. \par
First, let $N_{r}(S)$ be the tubular neighborhood of $S$ in $\overline{M_{i}}$, with radius $r$.  For any constant $\epsilon>0$, there exists a smooth function $\phi(x)$ supported in $N_{2\epsilon}(S)$, such that, $\phi(x)=1$ on $N_{\epsilon}(S)$. Meanwhile, we may find a smooth function $\tau (x)$ on $M^{0}_{i}$, satisfying the following:
\begin{enumerate}
\item $\tau (x)$ has support in the compact set $M^{0}_{i}\setminus (\overline{M_{i}}\setminus N_{4\epsilon}(S))$,
\item $\tau (x)=1$ on $M^{0}_{i}\setminus (\overline{M_{i}}\setminus N_{2\epsilon}(S))$
 \end{enumerate}\par
 Second, choose any smooth metric $g'$ over $M^{0}_{i}\setminus (\overline{M_{i}}\setminus N_{4\epsilon}(S))$. Define a new metric $g_{i}$ over $M^{0}_{i}$:
 
\begin{equation}
 g_{i}=\left\{  
\begin{array}{cc}
g,&\quad \overline{M_{i}}\setminus N_{4\epsilon}(S);\\
g',&\quad M^{0}_{i}\setminus M_{i};\\
(1-\phi)g+\tau g' &\quad \text{otherwise}.
\end{array}
\right.
\end{equation}
 where $g$ is a metric in the assumption of Theorem \ref{G}.\par
 
 If $[S]$ is a generator of $\pi_{2}(M)$, from the proof of Proposition \ref{J}, $[S]$ is contained in the kernel of the induced map $(f_{i})_{\ast}$. Then $\pi_{2}(M^{0}_{i})=\{0\}$ for each $i\in I$. Thanks to Lemma \ref{H}, $M^{0}_{i}$ is contractible. 
 Therefore, $(M^{0}_{i}, g_{i})$ satisfies the curvature condition as assumed  in Theorem \ref{G}. Due to Theorem \ref{A}, each $M^{0}_{i}$ is homeomorphic to $\mathbb{R}^{3}$. $M$ is the connected sum of two $\mathbb{R}^{3}$s. Hence, $M$ is diffeomorphic to $\mathbb{S}^{2}\times \mathbb{R}$.

If $[S]$ is not a generator, Proposition \ref{J} shows that the induced map  maps from $\pi_{2}(M)\rightarrow \pi_{2}(M^{0}_{i})$ is non-trivial. We see that $\pi_{2}(M^{0}_{i})$ is finite. In this case, there is a unique topological structure of $M^{0}_{i}$ as follows. By Proposition \ref{K} (in the following), each $M^{0}_{i}$ is homeomorphic to $\mathbb{R}^{3}$. That is to say, $M$ is homeomorphic to $\mathbb{S}^{2}\times\mathbb{R}$.
\end{proof}

 \begin{proposition}\label{K}Let $(M^{3}, g)$ be a simply-connected open 3-manifold satisfying that $\pi_{2}(M)$ is a finite group. For $0\in M$, assume that $r$ is the distance function from $0$. If there exists a number $\alpha\in [0,2)$, such that, 
$$\liminf\limits_{r(x) \rightarrow \infty }r^{\alpha}(x)\mathrm{Scal}(x)>0,$$then $M^{3}$ is homeomorphic to $\mathbb{R}^{3} $. 
\end{proposition}

\begin{proof}Suppose that $M$ is not homeomorphic to $\mathbb{R}^{3}$. In addition, we observe that $\pi_{2}(M)$ is non-trivial. (Otherwise, by Lemma \ref{H}, $M$ is contractible. Combined with Theorem \ref{A}, $M$ is $\mathbb{R}^{3}$, which contradicts our above assumption. )\par

By the sphere theorem, there exists a non-trivial embedded sphere $S$ cutting $M$ into two non relatively compact components. Two simply-connected non-compact manifolds $M^{0}_{0}$ and $M^{0}_{1}$ are obtained by the process described in the proof of Theorem \ref{G}. Therefore, $M^{0}_{0}$ and $M^{0}_{1}$ satisfy the following: 
\begin{itemize}
\item $M$ is the connected sum of $M^{0}_{0}$ and $M^{0}_{1}$
\item $|\pi_{2}(M^{0}_{i})|<|\pi_{2}(M)|$, since the map $\pi_{2}(M)\rightarrow \pi_{2}(M^{0}_{i})$ is surjective with non-trivial kernel (Proposition \ref{J}).
\end{itemize}
\par
We observe that one of $\{\pi_{2}(M^{0}_{i})\}_{i\in I}$ is nontrivial. (Otherwise, from Lemma \ref{H}, $M^{0}_{i}$ is contractible. From Theorem \ref{A}, $M^{0}_{i}$ is $\mathbb{R}^{3}$, which implies that $M$ is $\mathbb{S}^{2}\times \mathbb{R}$. This contradicts the assumption that $\pi_{2}(M)$ is finite. )\par 
Without loss of generality, we may assume that $\pi_{2}(M^{0}_{0})$ is nontrivial. Set $M_{1}=M^{0}_{0}$. $M^{1}_{0}$ and $M^{1}_{1}$ can be  constructed as the above process in the proof of Theorem \ref{G}. We may suppose that one of $\pi_{2}(M^{1}_{i})$ is non-trivial. (Otherwise, a similar argument for $\pi_{2}(M^{0}_{i})$ as above will work.) Then we can repeat this process, until the second homotopy groups of two new manifolds are both trivial. \par
After repeating the above process several times, three families of manifolds $M_{k}$, $M^{k}_{0}$ and $M^{k}_{1}$ are constructed as in the above process. These manifolds satisfy the following: 
\begin{itemize}
\item $M_{k}$ is the connected sum of $M^{k}_{0}$ and $M^{k}_{1}$,
\item $|\pi_{2}(M_{k})|<|\pi_{2}(M_{k-1})|$ .
\end{itemize}  
Because $\pi_{2}(M)$ is finite and $|\pi_{2}(M_{k})|<|\pi_{2}(M_{k-1})|$, this process will stop after finite steps. There exists an integer $k_{0}$ such that $\pi_{2}(M_{k_{0}})$ is nontrivial and $\pi_{2}(M^{k_{0}}_{i})$ is trivial for each $i$. From Lemma \ref{H}, this implies that $M^{k_{0}}_{i}$ is contractible. Due to Theorem \ref{A}, $M^{k_{0}}_{i}$ is homeomorphic to $\mathbb{R}^{3}$. However, $M_{k_{0}}$ is the connected sum of $M^{k_{0}}_{0}$ and $M^{k_{0}}_{1}$. Hence, $M_{k_{0}}$ is $\mathbb{S}^{2}\times \mathbb{R}$, which contradicts the fact that $\pi_{2}(M_{k_{0}})$ is finite.\par
\end{proof}

\section*{Acknowledgement}
The author would like to express his deep gratitude to his advisor, Professor G\'erard Besson for his useful suggestions, also for his constant encouragement and support. The work is supported by ERC GETOM project and ANR-CCEM project.

\end{document}